\documentclass[a4paper,11pt, reqno]{amsart}
\usepackage[DIV=12, oneside]{typearea} 
\usepackage[utf8]{inputenc}
\usepackage[T1]{fontenc}
\usepackage[english]{babel}
\usepackage{graphicx,amssymb, verbatim,enumerate,amsmath,mathtools}
\usepackage{esint}
\usepackage{amssymb, amsthm,lineno}
\usepackage{enumerate, mathrsfs, hyphenat, graphicx}
\usepackage[font=sf, labelfont=sf]{caption}
\usepackage{thmtools, thm-restate}
\usepackage{thm-patch, thm-kv, thm-autoref}
\usepackage[colorlinks=true, urlcolor=blue, linkcolor=blue, citecolor=green]{hyperref}
\theoremstyle{plain}
\declaretheorem[title=Theorem, parent=section]{sa}
\declaretheorem[title=Lemma,sibling=sa]{lem}
\declaretheorem[title=Corollary,sibling=sa]{cor}
\declaretheorem[title=Proposition,sibling=sa]{prop}
\newtheorem*{thm*}{Satz}

\theoremstyle{definition}

\numberwithin{equation}{section}
\renewcommand{\L}{\mathcal{L}}
\newcommand{\R}{  \mathbb{R}}
\newcommand{\C}{  \mathcal{C}}

\newcommand{\N}{  \mathbb{N}}

\newcommand{\DC}{ {}^{C}\D_{a}^{\alpha}}
\newcommand{\AB}{ {}^{ABC}\D_{a}^{\alpha}}
\newcommand{\fiab}[2]{\prescript{\text{AB}}{#1} I_{t}^{#2}}

\newcommand{\K}{{\mathcal K}}
\newcommand{\M}{{\mathcal M}}

\renewcommand{\epsilon}{\varepsilon}

\newcommand{\bet}[1]{\left| #1 \right|}

\addto\extrasenglish{
\renewcommand{\sectionautorefname}{Section}%
}

\newtheorem{thm}{Theorem}[section]

\newtheorem{rem}[thm]{Remark}

\def\A{{\mathcal A}}

\def\D{{\mathcal D}}

\def\J{{\mathcal J}}
\def\K{{\mathcal K}}
\def\L{{\mathcal L}}
\def\M{{\mathcal M}}

\def\T{{\mathcal T}}
\def\real{{\mathbb R}}

\def\al{\alpha}

\def\intersect{\bigcap}

\newcommand{\abs}[1]{\left| #1 \right|}

\usepackage[usenames,dvipsnames]{xcolor}

\begin{document}\renewcommand{\sectionautorefname}{Section}
\renewcommand{\sectionautorefname}{Section}
\title{Parabolic problem with fractional time derivative with nonlocal and nonsingular  Mittag-Leffler kernel}
\author{J.D~Djida}
\author{A.~Atangana} 
\author{I.~Area}

\begin{abstract}
We prove H\"older regularity results for a class of nonlinear parabolic problem with fractional-time derivative with nonlocal and Mittag-Leffler nonsingular kernel. Existence of weak solutions via approximating solutions is proved. Moreover, the H\"{o}lder continuity of viscosity solutions is obtained. We get the similar results as those obtained by Allen (see {\url{https://arxiv.org/abs/1610.10073}}).
\end{abstract}

\address[Djida]{African Institute for Mathematical Sciences (AIMS), P.O. Box 608, Limbe Crystal Gardens, South West Region, Cameroon. }
\email[Djida]{jeandaniel.djida@aims-cameroon.org}

\address[Atangana]{Institute for Groundwater Studies, Faculty of Natural and Agricultural Sciences, University of the Free State, 9301, Bloemfontein, South Africa.}
\email[Atangana]{abdonatangana@yahoo.fr}

\address[Area]{Departamento de Matem\'atica Aplicada II,
E.E. Aeron\'autica e do Espazo, Universidade de Vigo,
Campus As Lagoas s/n, 32004 Ourense, Spain.}
\email[Area]{area@uvigo.es}


\maketitle

\pagestyle{headings}		
\markboth{}{H\"older regularity for parabolic problem with general class of nonlocal space-time operators}

\section{Introduction}\label{sec:intro}

In this work we prove some H\"older regularity results for viscosity solutions of integro-differential equations in which the kernels defining the fractional time operators belongs to a more general class with nonlocal and Mittag-Leffler nonsingular kernel and the spatial nonlocal operator kernel corresponds to the fractional Laplacian. One of our aims is to bring into better alignment this analysis of H\"older with fractional time derivative involving the Mittag-Leffler function with the results which have been known \cite{Allen2, Allen1} where the fractional time derivative in the sense of Caputo were used . \medskip

We first mention that we have tried to collect the notation contained herein-- as much as possible-- in \autoref{Sec:Notation}.  We also include a discussion of background in \autoref{sec:Background}.\medskip

The specific equation which is of interest is
\begin{equation} \label{equation:problem}
\L u(t,x) - \mathcal{J}u(t,x) = g(t,x), \quad \text{in} ~~(-\infty,b)\times \R^{n},
\end{equation} 

where the nonlocal spatial operator $\mathcal{J}$~\cite{Caffarelli2, Fall} we make is 
\begin{equation}\label{IntroEq:LinearK}
\J u(x) = \int_{\R^{n}} \delta_h u(x) \K(x,h) dh.
\end{equation}
$\delta_h u(x)=u(x+h)+u(x-h)-2u(x)$ denote the second order difference relation. We choose for $\sigma(0,2)$ the kernel 
\[
\K(x,h) = C(n,\sigma)\abs{h}^{-n-\sigma},
\]
which for an appropriate constant $C(n,\sigma)$ gives $\mathcal{J} = (-\Delta)^{\sigma/2}$.\medskip

Instead of considering the class of kernels  $\K$ as it has been shown in \cite{Kassmann}, we consider the the simplest one which allowed to have large regions where $\K(x,h)$ is comparable to $\abs{h}^{-n-\al}$ from below. This class of kernel that is called $\A_{sec}$ (see~ \eqref{PrelimNotationEq:ClassAsect}), are treat without assuming any regularity in the $x$ variable through the assumption that they need only to be above $\abs{h}^{-n-\sigma}$ on a possibly set as
\begin{align}\label{specialK}
\frac{\lambda}{|h|^{n+\sigma}}\leq \K(x,h)\leq \frac{\Lambda}{|h|^{n+\sigma}} \quad h\in\R^n\setminus\{0\}.
\end{align}

Furthermore, for $\alpha \in (0,1)$ for all $t<a$ and for $n\in \N$, we defined the fractional time-operator $\L$ involved in \eqref{equation:problem} as 
\begin{equation}\label{Fract_derivative}
\L u(t) = \mathcal{C}(n,\alpha)\int_{-\infty}^{t} \big[u(t)-u(s)\big]\T(t,s)ds,
\end{equation}
where the kernel $\T(t,s)$ belongs to the class of kernels $\T_{sec}$ described by 
\begin{equation}\label{eq:kernel_relation}
\begin{aligned}
&\T_{sec} = \bigg\{ \T: (-\infty,b) \to \R: \T(t,t-s) = \T(t+s,t), ~~\text{and} \\
&-c\lambda \frac{(t-s)^{\alpha-1}}{\Gamma(\alpha+1)} \leq \T(t,s) \leq -c\Lambda \frac{(t-s)^{\alpha-1}}{\Gamma(\alpha+1)}  \bigg\}.
\end{aligned}
\end{equation}

\begin{rem}
This class of kernel contains a number of specific kernels such as the ones involved in Marchaud derivative \cite{samko1993fractional}, Caputo derivative \cite{Allen2, samko1993fractional, Ivan, Rico} and the fractional derivative with nonlocal and nonsingular Mittag-Leffler kernel \cite{atanbaleanu} if we consider the fractional time derivative in the sense of Caputo, defined as 
\[
\DC u(t) = \frac{\alpha}{\Gamma(1-\alpha)}\int_{-\infty}^{t}\big[u(t)-u(s)\big](t-s)^{-\alpha-1}ds,
\]
the kernel $(t-s)^{-\alpha-1}$ belongs the class of kernels $\T_{sec}$. We recall that this Caputo derivative has been used in \cite{Allen2, Allen} to prove H\"{o}lder regularity of some parabolic problems in the non-divergence form.\medskip

Also, for the fractional time-derivative with nonlocal and nonsingular Mittag-Leffler kernel~\cite{atanbaleanu}
\begin{equation}\label{eq:Atangana_Baleanu_other_form1}
\AB u(t) := c_{\alpha}\int_{-\infty}^t \big[u(t)- u(s)\big] (t-s)^{\alpha-1} E_{\alpha,\alpha}\big[c \big(t - s\big)^{\alpha} \big]\ ds.
\end{equation}
The kernel $(t-s)^{\alpha-1} E_{\alpha,\alpha}\big[c \big(t - s\big)^{\alpha} \big]$
belongs the class of kernels $\T_{sec}$ as well. \medskip

Further properties of the fractional time derivative will be given in \autoref{sec:prelim}.
\end{rem}

Besides the mathematical satisfactions of the Caputo fractional derivative, the attentiveness for the derivative with nonlocal and nonsingular Mittag-Leffler kernel is based on its properties of portraying the behavior of orthodox viscoelastic materials, thermal medium, material heterogeneities and some structure or media with different scales~\cite{atanbaleanu}. The nonlocality of the new kernel allows a good description of the memory within structure and media with different scale, which cannot be described by classical fractional derivative. This derivative also takes into account power and exponential decay laws which, many natural occurrence follow. The one with the Mittag-Leffler functions allows us to describe phenomena in processes that progress or decay too slowly to be represented by classical functions like the exponential function and its successors. The Mittag-Leffler function arises naturally in the solution of fractional integral equations, and especially in the study of the fractional generalization of the kinetic equation, random walks, L\'{e}vy flights, and so called super-diffusive transport. \medskip 

\begin{rem}
There are various definitions---Riemann, Liouville, Caputo, Grunwald-Letnikov, Marchaud, Weyl, Riesz, Feller, and others--- for fractional derivatives and integrals, (see e.g. \cite{samko1993fractional, Rico, Richard, Hilfer, MR0361633} and references therein). This diversity of definitions is due to the fact that fractional operators take different kernel representations. So instead of studying the parabolic problem for different kernels, we propose in this note to study the problem for a large class of kernels $\T_{sec}$ satisfying the properties given by \eqref{eq:kernel_relation}.
\end{rem}

Of interest is the regularity of solutions to these nonlocal equations involving the fractional time derivative and the nonlocal spacial operator, some authors as \cite{Rico, Allen, Luis_sylvestre1, Hector} studied the problem of H\"older continuity for solutions to master equations and H\"older continuity for parabolic equations with Riemman derivative and the Caputo fractional time derivative and divergence form nonlocal operator. Furthermore very recently the author in \cite{Allen2} proves the H\"older continuity of viscosity solutions to \eqref{equation:problem} in the non divergence form, but using the generalized fractional time derivative of Marchaud or Caputo type under the appropriate assumptions. They obtained an estimate which remains uniform as the order of the fractional derivative $\alpha \to 1$. \medskip

With the clear intention to obtain a similar result as in \cite{Allen2, Allen1} where the generalized Marchaud or Caputo derivative was used to show the  H\"older continuity of viscosity solutions, we propose to study the same problem, but now with another kernel (e.g.~ $\T = (t-s)^{\alpha-1} E_{\alpha,\alpha}\big[c \big(t - s\big)^{\alpha} \big]$ ) where the Caputo kernel can be deduce from and which belongs to the class of kernels $\T_{sec}$. \medskip

\begin{rem}
Despite the fact that we obtain similar results as those already presented in the literature, we consider that this approach with more general kernel is a useful contribution to the field since up to now (to the best of our knowledge) most of the kernels of fractional time derivative in the literature belong to $\T_{sec}$. For example if we take the case of Caputo fractional derivative, it has been proven in~ \cite{atanbaleanu} that there exists a well defined function $v$ that equalizes the Caputo fractional time-derivative and the fractional time derivative with nonlocal and nonsingular Mittag-Leffler kernel with fractional order $\alpha$
\[
\DC v(t) = \AB v(t).
\]
The solution is obtained by applying the Fourier transform in both side. Hence the result as
\[
v(t) = v(0) \big(1 - k_{\alpha}\big) t^{\alpha-1} E_{\alpha, \alpha} \big[(c + k_{\alpha}) t^{\alpha} \big] + 
\frac{cv(0)}{\Gamma(\alpha)} \int_{0}^{t}s^{\alpha-1} \big(t-s\big)^{\alpha-1}E_{\alpha, \alpha} \big[(c + k_{\alpha}) s^{\alpha} \big]ds, 
\]
where $ k_{\alpha} = \frac{B(\alpha)}{1-\alpha}$.
\end{rem}

To study regularity properties of solutions to equation~\eqref{equation:problem}, one could in some sense study the solution $u$ which simultaneously solve the two inequalities
\[
\inf_{\T \in \T_{sec}}\bigg\{\L_{\T}u(t,x)-\J u(t,x)\bigg\} \leq C \ \text{and}\ \sup_{\T \in \T_{sec}}\bigg\{\L_{\T} u(t,x)-\J u(t,x)\bigg\} \geq -C \ \textnormal{in}\ (\infty,T) \times \Omega.
\]
The kernel $\T = (t-s)^{\alpha-1} E_{\alpha,\alpha}\big[c \big(t - s\big)^{\alpha} \big]$ is chosen in the class of kernels $\T_{sec}$ which at least contains all the kernels involved in the fractional time derivatives in the literature so that it will be convenient if one wishes to attain further properties of the extremal operators.\medskip

The program of studying existence of solutions and regularity properties of parabolic problem with fractional nonlocal space-time operators such as \eqref{equation:problem} was presented in \cite{Allen2, Rico, Allen}, using respectively Riemann fractional derivative and Caputo fractional derivative. We extend those results to cover the larger class, $\T_{sec}$.  Our main results are existence of weak solution and H\"{o}lder regularity estimate.

\begin{sa}[Existence of weak solutions] \label{theo:existence}
Let $\vartheta$ be a bounded an Lipschitz function on $(-\infty,b)\times \R^{n}$ and assume that the function $g$ is regular enough. So for a given smooth bounded initial data $u_{0}>0$, there exists a weak solution $u$ to the weak formulation 
\[
\begin{aligned}
& c_{\alpha} \int_{\R^{n}}\int_{-\infty}^{b}\int_{-\infty}^{t} \big[u(t,x)-u(s,x)\big]\big[\vartheta(t,x)-\vartheta(s,x)\big]\T(t,s,x)ds~dt~dx \nonumber \\
&+ \int_{a}^{b} \int_{\R^{n}}\int_{\R^{n}} \K(t,x,\xi)\big[u(t,x)-u(t,\xi) \big] \big[\vartheta(t,x)-\vartheta(t,\xi) \big] dx ~d\xi~dt \nonumber \\
&+ c_{\alpha} \int_{\R^{n}} \int_{-\infty}^{b}\int_{-\infty}^{2t-b} u(t,x)\vartheta(t,x)\T(t,s,x)ds~dt~dx \nonumber 
- \int_{\R^{n}}\int_{-\infty}^{b} u(t,x)~ \L \vartheta(t,x) dt~dx \\
&= \int_{-\infty}^{b} \int_{\R^{n}} f(t,x)\vartheta(t,x)dx~dt.\nonumber
\end{aligned}
\]
\end{sa} 

Next we state the  H\"{o}lder regularity estimate result.\medskip

\begin{sa}[H\"older Regularity]\label{thm:holder}
Let $\alpha \in(0,1)$, $\sigma \in(\sigma_{0},2)$ and let $\M^\pm_{\A}$ be as defined in \eqref{PointEq:FormulaMPlus}) and \eqref{PointEq:FormulaMMinus}). Assume also that $g \in L^{\infty} (-\infty, b) \times \R^{n}$. There are positive constants $\beta \in (0,1)$ and $C \geq 1$ depending only on $n,\lambda,\Lambda,\alpha, \sigma$ such that if $u$ is bounded continuous viscosity solution in $B_{2} \times [-2,0]$ satisfying 
\begin{equation}\label{IntroEq:ExtremalBounded}
\L_{\T}u - \M^-_{\A}u\leq \varepsilon_{0} \ \ \text{and}\ \L_{\T}u - \M^+_{\A}u \geq -\varepsilon_{0},
\end{equation}
then $u$ is  H\"older continuous in $B_{1}\times [-1,0]$ and for $(x,t),(y,s) \in B_{1}\times [-1,0]$ the following estimates holds
\begin{align}\label{c_alpha_estimate}
|u(x,t)-u(y,s)| \leq C (\| u\|_{L^{\infty}} + \varepsilon_0^{-1} \| g\|_{L^{\infty}})|x-y|^{\kappa}+|t-s|^{\kappa \alpha/(2\sigma)}.
\end{align}
Furthermore $C$ remains bounded as $\alpha \to 1$ and $\sigma\to 2$.
\end{sa}

\begin{rem}
Notice that in order to get solutions to \eqref{thm:holder}, the solution of fractional differential equations involving the fractional time derivative with nonlocal and nonsingular Mittag-Leffler kernel is proposed in order to use it as a test function for of viscosity solution. The case with the Caputo derivative was well presented by the author in \cite{Allen2} showing  that if $|\{x \in B_1\times (-2,-1): u(x,t)\leq 0\}|\geq \mu_1$, then $u(t)\geq \mu_2$ if $t \in (-1,0)$. \medskip
 
So we omit to prove it in this note in the case for the fractional time derivative with nonlocal and non singular Mittag-Leffler kernel since the idea of the proof is similar.\medskip

So with the H\"older continuity estimates for ordinary differential equations involving the Caputo derivative \cite{Allen2}, our class of weak solutions will be considered in the viscosity sense as described in  \autoref{sec:PointWiseEvaluation}. This will then allow us to get the similar result. 
\end{rem}

The organization of the article is as follows.  In \autoref{sec:Background} we review some background related to \autoref{thm:holder}.  In \autoref{sec:prelim} we collect notation, definitions, and preliminary results regarding (\autoref{theo:existence}) and \autoref{thm:holder}. \autoref{sec:Existence} is dedicated to the sketch of proof of the existence of weak solutions via approximating solutions, mainly \autoref{theo:existence}. Finally in \autoref{sec:PointWiseEvaluation} we discuss the pointwise estimates and put together the remaining pieces of the proof for the and H\"{o}lder Regularity.

\section{Background}\label{sec:Background}

There are few collection of results related to \autoref{thm:holder} with both space and time fractional nonlocal operators.  We will focus on the type of results which only depend on the ellipticity constants, $\lambda$ and $\Lambda$, for the spacial nonlocal operator as well as possibly the order, $\alpha$, $\sigma$ and we shall try to see if we recover the results of Mark Allen et al. \cite{Allen2, Allen1, Allen} by using fractional time derivative with nonlocal and nonsingular Mittag-Leffler kernel.

The problem of regularity of parabolic problem with fractional time derivative in time are new in the literature. As we can notice in the recent article of \cite{Caffarelli}, where the authors used the original method of De Giorgi to prove boundedness of solutions and local H\"older regularity. Similar approach to prove apriori local H\"older estimates of solutions to the fractional parabolic type equation was also been used in \cite{Allen1}, where they have followed the De Giorgi method as in \cite{Caffarelli} but now by taking into account the fractional time 
 derivative in the sense of Caputo \cite{samko1993fractional}. As an earlier result, the used of the fractional nature of the derivative made the estimates to not remain uniform as the fractional order $\alpha \to 1$. \medskip
 
In the same direction, we shall also bring the attention of the reader on the fact that similar result has been studied by Zacher in \cite{Rico} but instead of using Caputo derivative the author used the Riemann-Liouville fractional time derivative and zero right hand side. \medskip

One should notice that these results based on the fractional space and time where obtained both kernels of these kernel are bounded. In the case of the fractional spatial nonlocal operator,there are a few interesting distinctions that are usually made: whether or not $\K(x,h)$ is assumed to be even in $h$; whether or not the corresponding equations are linear; whether or not a Harnack inequality holds \cite{Kassmann}.\medskip

Regularity results (such as \autoref{thm:holder}) as well as the Harnack inequality for linear equations with operators similar to (\ref{IntroEq:LinearK}) were obtained \cite{ Kassmann, Luis_sylvestre}. Furthermore in \cite{Allen2} H\"older continuity of viscosity solutions to certain nonlocal parabolic equations that involve a generalized fractional time derivative of Marchaud or Caputo type as well is obtained under the assumption that kernel of the fractional time operator satisfied the symmetry condition $\T(t,t-s) = \T(t+s,t)$. The estimates are uniform as the order of the operator $\alpha$ approaches $1$, so that the results recover many of the regularity results for local parabolic problem.\medskip

Finally, higher regularity in time type estimates were obtained in \cite{Allen1}. 

An important class of kernels are those for which the symmetry $\T(t,t-s) = \T(t+s,t)$ is assumed to hold and for which all the fractional time operators in the literature belongs to the class $T_{sec}$. This class could be also extended to the non-symmetric case. We will discuss in this note the case where the kernel of the nonlocal time operator belongs to a more general class of kernels $T_{sec}$ which contains must of the  kernels known in the literature. 

\section{Preliminaries}\label{sec:prelim}

\subsection{Notation}\label{Sec:Notation}
We first collect some notations which will be used throughout this article.
\allowdisplaybreaks
\begin{itemize}
\item $\L$ - the nonlocal fractional time-derivative with nonsingular Mittag-Leffler kernel.
\item $\sigma \in(\sigma_{0},2)$ - denote the order of the nonlocal spatial operator. 
\item $\alpha$ - will always denote the order of the space-time arbitrary order derivative.
\begin{equation} \label{PrelimNotationEq:ClassAsect}
\A_{sec} = \bigg\{K:\R^n\to\R\ :\ \J(-h)=\J(h),\quad  \text{and}\quad \frac{\lambda}{|h|^{n+\sigma}}\leq \J(h)\leq  \frac{\Lambda}{|h|^{n+\sigma}}\bigg\}
\end{equation}
\begin{equation}\label{PrelimNotationEq:SecondDiffDef}
\begin{aligned}
&\delta_hu(x) = u(x+h)+u(x-h)-2u(x)
\end{aligned}
\end{equation}
\begin{equation}
\begin{aligned}
\J_{\A}u(x) &= \int_{\real^n}\delta_h u(x) \J(h) dh \nonumber\\
\mu(dh) &= \abs{h}^{-n-2\sigma}dh \nonumber\\
Q_{\varepsilon}(x_0) &=\left\{x\in\R^n\, :\, \bet{x-x_0}_\infty<\tfrac{\varepsilon}{2}\right\}\nonumber\\
B_{\varepsilon}(x_0)&=\left\{x\in\R^n\, :\, \bet{x-x_0} < \varepsilon \right\}\nonumber
\end{aligned}
\end{equation}
\end{itemize}

We use $\bet{\cdot}$ for the absolute value, the Euclidean norm, and the $n$-dimensional Lebesgue measure at the same time. Throughout this article $\Omega\subset\R^n$ denotes a bounded domain.
For cubes and balls such that $x_0=0$ we write $Q_l$ instead of $Q_l(0)$ and similarly for $B_{l}$. The following hold:
\[ B_{1/2} \subset Q_1 \subset Q_3 \subset B_{3} \subset B_{2} \,. \]

\subsection{Definitions}
We use the definitions and basic properties the of the fractional time derivative with nonlocal and nonsingular Mittag-Leffler kernel and of viscosity solutions from \cite{atanbaleanu} and \cite{Allen2}, and for H\"{o}lder continuity \cite{Allen1, Allen, Luis_sylvestre, Vasquez}. 

\subsection{The fractional time derivative with nonlocal and nonsingular Mittag-Leffler kernel}\label{Sec:2}
In this section for the convenience of the reader, we recall some definitions of fractional time derivative with the nonlocal and nonsingular kernel as stated in \cite{atanbaleanu} and state some of its new properties.\medskip

The fractional time derivative with nonlocal and Mittag-Leffler nonsingular kernel, recently introduced by Atangana and Baleanu and known as the Atangana-Baleanu derivative is useful in modelling equations arising in porous media. One formulation of the Atangana-Baleanu derivative is 
\begin{equation}\label{Atangana_baleanu_def}
\AB u(t) =  \frac{B(\alpha)}{1 - \alpha}\int_{a}^{t} E_{\alpha}\big[-c \big(t - s\big)^{\alpha } \big]u'(s)ds.
\end{equation}
The associate integral of the fractional time derivative with nonlocal and Mittag-Leffler nonsingular kernel, is defined as
\begin{equation}\label{eq:abfi}
\fiab{a}{\alpha} u(t) = \frac{1-\alpha}{B(\alpha)}u(t) + \frac{\alpha}{B(\alpha) \Gamma(\alpha)} \int_{a}^{t} u(y) (t-y)^{\alpha-1} dy.
\end{equation}
In the above formulas $B(\alpha )$ is a constant depending on $\alpha$ such that
\[
B(\alpha)=1-\alpha +\frac{\alpha }{\Gamma (\alpha )},\quad c = -\frac{\alpha}{1-\alpha},
\]
and $E_{\alpha, \beta}(z)$ is the two-parametric Mittag-Leffler function defined in terms of a series as the following entire function as
\[
E_{\alpha ,\beta}(z)=\sum_{k=0}^{\infty} \frac{z^{\alpha k}}{\Gamma(\alpha  k + \beta)},~~\text{and}~~ \quad \beta>0, ~ \quad z \in \mathbb{C}.
\] 
   
The others representation of the Atangana-Baleanu fractional time derivative in the sense of Caputo holds pointwise 
\begin{multline}\label{Atangana-Baleanu_Caputo_other}
\AB u(t) = \nu_{\alpha} E_{\alpha }\big[-c \big(t - a\big)^{\alpha } \big] \big[u(t)- u(a)] \\
 + c_{\alpha}\int_{a}^{t}(t- s)^{\alpha-1} E_{\alpha,\alpha}\big[-c \big(t - s\big)^{\alpha} \big] \big[u(t)- u(s)\big]ds,
\end{multline}
where $\nu_{\alpha} = B(\alpha)(1-\alpha)^{-1} $ and $c_{\alpha} = -c \nu_{\alpha}$. We set 
\begin{equation}\label{eq:kernel}
\T(t,s) = (t-s)^{\alpha-1} E_{\alpha,\alpha}\big[-c \big(t - s\big)^{\alpha} \big],
\end{equation}
as our bounded kernel in time, satisfying the relation $\T(t,t-s) = \T(t+s,t)$ and
\begin{equation}\label{eq:kernel_relation1}
-c\lambda \frac{(t-s)^{\alpha-1}}{\Gamma(\alpha+1)} \leq \T(t,s) \leq -c\Lambda \frac{(t-s)^{\alpha-1}}{\Gamma(\alpha+1)}.
\end{equation}
  
In this setting, following the idea of \cite{Allen1, Allen} we define $u(t)=u(a)$ for $t<a$,
\begin{equation}  \label{eq:Atangana_Baleanu_other_form}
\L u(t) := c_{\alpha}\int_{-\infty}^t \big[u(t)- u(s)\big] \K(t,s) \ ds.
\end{equation}
One of the immediate consequence of the formulation \eqref{eq:Atangana_Baleanu_other_form} is that, it allows to drop out data and it is also useful for viscosity solutions~\cite{Allen}.\medskip

The notion of viscosity solution and supersolutions for the initial values problem involving the time derivative with the non singular Mittag-Leffler kernel in \eqref{eq:Atangana_Baleanu_other_form} is similar to the one of \eqref{equation:problem}. For this purpose we state the following Proposition which shows that \eqref{eq:Atangana_Baleanu_other_form} is well defined.

\begin{prop}\label{continuous_bounded_derivative}
Let $u$ a continuous bounded function and $w \in \mathcal{C}^{0,\beta}$ with $\alpha < \beta \leq 1$. If $w \geq (\leq) u$ on $[t_{0}-\varepsilon,t_{0} ]$ and $w(t_{0}) = u(t_{0})$, then the integral
\[
c_{\alpha}\int_{-\infty}^{t_{0}} \big[ u(t_{0}) -u(s) \big] \T(t,s) \ ds
\]
is well defined, so that $\L u(t_{0})$ is well defined.
\end{prop} 
\begin{proof}
Let assume that $w \geq u$ on $[t_{0}-\varepsilon,t_{0} ]$. So it comes from \eqref{eq:Atangana_Baleanu_other_form}
\[
\begin{aligned}
& c_{\alpha} \int_{-\infty}^{t_{0}} \big[ u(t_{0}) -u(s) \big] \T(t,s) \ ds = \\ & c_{\alpha} \int_{-\infty}^{t_{0}-\varepsilon} \big[ u(t_{0}) -u(s) \big] \T(t,s) \ ds + c_{\alpha} \int_{t_{0}-\varepsilon}^{t_{0}} \big[ u(t_{0}) -u(s) \big] \T(t,s) \ ds \\
& \geq c_{\alpha} \int_{-\infty}^{t_{0}-\varepsilon} \big[ u(t_{0}) -u(s) \big] \T(t,s) \ ds + c_{\alpha} \int_{t_{0}-\varepsilon}^{t_{0}} \big[ w(t_{0}) - w(s) \big] \T(t,s) \ ds \\
&\geq c_{\alpha} \int_{-\infty}^{t_{0}-\varepsilon} \big[ u(t_{0}) -u(s) \big] \T(t,s) \ ds - cc_{\alpha} \Lambda \big\|w \big\|_{\C^{0,\beta}} \int_{t_{0}-\varepsilon}^{t_{0}} \frac{\big(t_{0}-s\big)^{\alpha+\beta -1} }{\Gamma(1+\alpha)} \bigg) ds \\
& \geq c_{\alpha} \int_{-\infty}^{t_{0}-\varepsilon} \big[ u(t_{0})-u(s) \big] \T(t,s) ds + c c_{\alpha} \Lambda \frac{\varepsilon^{(\alpha+\beta)}}{\big(\alpha+\beta \big)\Gamma(1+\alpha)}\big\|w \big\|_{\C^{0,\beta}} \\
&\geq c_{\alpha}\frac{c \Lambda \varepsilon^{\alpha+\beta}}{\big(\alpha+\beta \big)\Gamma(1+\alpha)} \bigg(2\big\|u \big\|_{L^{\infty}} + \big\|w \big\|_{\C^{0,\beta}}\bigg)\\
&\geq \tilde{A}_{\alpha,\beta}.
\end{aligned}
\]
Hence the integral is well defined.
\end{proof}

Next we state the following estimate of bound of the fractional time derivative $\L \varrho(t)$ that will be useful for the proof of H\"{o}lder continuity in \autoref{sec:PointWiseEvaluation}.
\begin{prop}\label{proposition_estimate}
Let us consider $\varrho(t)= \max \{2|rt|^{\nu}-1, 0 \}$ under the condition that $\nu <\alpha$, and $r= \min\{4^{-1},4^{-\alpha/2\sigma}\}$.\medskip
 
If $t_1\leq 0$ then 
\[
- d_{\alpha, \nu} \leq \L \varrho(t_{1}) \leq 0 
\]
where the constant $d_{\alpha,\nu}$ depends on $\alpha$ and $\nu$.
\end{prop}
\begin{proof}
From \eqref{Atangana_baleanu_def} and \eqref{eq:Atangana_Baleanu_other_form} the rescaled fractional time derivative take the form
\[
-\nu \int_{a}^{t} E_{\alpha}\big[c\big(t - s\big)^{\alpha } \big]|rs|^{\nu - 1}ds \geq -\nu \int_{-\infty}^{t} E_{\alpha}\big[c \big(t - s\big)^{\alpha } \big]|rs|^{\nu - 1}ds.
\]
One notice that $|rs|^{\nu - 1}$ and the Mittag-Leffler function $E_{\alpha}\big[c\big(t - s\big)^{\alpha } \big]$ are increasing function of $s$, if $s < 0$, so 
\[
\nu \int_{-\infty}^{t} E_{\alpha}\big[c\big(t - s\big)^{\alpha } \big]|rs|^{\nu - 1}ds
\]
is an increasing function of $t$. \medskip

Furthermore, if $t \leq -1$, then if follows that 
\[
\L \varrho \geq -\nu \int_{-\infty}^{t} E_{\alpha}\big[c \big(t - s\big)^{\alpha } \big]|rs|^{\nu - 1}ds \geq -\nu \int_{-\infty}^{-1} E_{\alpha}\big[c \big(-1 - s\big)^{\alpha } \big]|rs|^{\nu- 1}ds \geq - d_{\alpha, \nu}.
\]
Now if $t> -1$, then 
\[
\L \varrho \geq -\nu \int_{-\infty}^{-1} E_{\alpha}\big[c \big(t - s\big)^{\alpha } \big]|rs|^{\nu - 1}ds \geq -\nu \int_{-\infty}^{-1} E_{\alpha}\big[c \big(-1 - s\big)^{\alpha } \big]|rs|^{\nu - 1}ds \geq - d_{\alpha, \nu}.
\]
\end{proof}

Next, in order to get informations over the various time, we are interested on the solution to the fractional differential equation involving the Atangana-Baleanu fractional derivative, in the form
\begin{equation} \label{equation:eqdiff}
\L u(t) = -c_{1}~u(t) + c_{0}~h(t), \quad c_{1},~c_{0} < \infty. 
\end{equation}

\begin{prop}\label{proposition_solution}
Let $u \in H^{1}(0,b)$, $b>0$, and $h \in \mathcal{C}^{2}$, such that the Atangana-Baleanu fractional derivative exists. Then, the solution of differential equation \eqref{equation:eqdiff}, for $c_{1}=0$, is given by
\begin{equation}\label{eq:sol1}
 u(t) = \frac{(1-\alpha)c_{0}}{B(\alpha)}h(t) + \frac{\alpha c_{0}}{B(\alpha) \Gamma(\alpha)} \int_{0}^{t} h(s) (t-s)^{\alpha-1} ds,
\end{equation}
and for $c_{1}\neq 0$ by
\begin{equation}\label{eq:sol2}
\begin{aligned}
u(t) & = \zeta E_{\alpha}\big[-\gamma t^{\alpha}\big]u(0)\\
&+ \frac{\alpha c_{0} \zeta}{B(\alpha)} \int_{0}^{t}\bigg( E_{\alpha,\alpha}\big[-\gamma(t-s)^{\alpha}\big] + \frac{(1-\alpha)}{\alpha}\gamma^{-2\alpha}E_{\alpha,\alpha}\big[-\gamma^{-2\alpha}(t-s)^{\alpha}\big]\bigg)(t-s)^{\alpha-1}h(s)ds,
\end{aligned}
\end{equation} 
with $\gamma = \frac{\alpha c_{1}}{\big( B(\alpha) + (1-\alpha)c_{1} \big)}$ and $\zeta = \frac{B(\alpha)}{\big( B(\alpha) + (1-\alpha)c_{1} \big)}$.\medskip
\end{prop}
\begin{proof}
For $c_{1}=0$, it is obvious that the result is 
\[
 u(t) = \frac{1-\alpha}{B(\alpha)}h(t) + \frac{\alpha c_{0}~}{B(\alpha) \Gamma(\alpha)} \int_{0}^{t} h(s) (t-s)^{\alpha-1} ds.
\]

We purchase the proof of \autoref{proposition_solution} for  $c_{1} \neq	0 $ by using \eqref{eq:abfi} and by applying the Laplace transform on the equation \eqref{equation:eqdiff}. It comes that
\begin{equation}
\begin{aligned}\label{eq:aux2}
\fiab{0}{\alpha} \left \{ \AB u(t) \right \} & = \fiab{0}{\alpha} \big\{-c_{1}~u(t) + c_{0}~h(t) \big\}\\
u(t)-u(0) &= -c_{1} \bigg\{ \frac{1-\alpha}{B(\alpha)}u(t) + \frac{\alpha}{B(\alpha) \Gamma(\alpha)} \int_{0}^{t} u(s) (t-s)^{\alpha-1} ds \bigg\} \\
&+ c_{0}\bigg\{ \frac{1-\alpha}{B(\alpha)}h(t) + \frac{\alpha}{B(\alpha) \Gamma(\alpha)} \int_{0}^{t} h(s) (t-s)^{\alpha-1} ds \bigg\}
\end{aligned}
\end{equation}
If we apply the Laplace transform to both sides of \eqref{eq:aux2}, it yields
\begin{multline}
\bigg[ \frac{B(\alpha)+(1-\alpha)c_{1}}{B(\alpha)} + \frac{\alpha c_{1}}{B(\alpha)}\frac{1}{p^{\alpha}} \bigg]u(p) = \frac{1}{p}u(0) + \frac{c_{0}(1-\alpha)}{B(\alpha)}h(p) + \frac{\alpha c_{0}}{B(\alpha)\Gamma(\alpha)}\Gamma(\alpha)p^{-\alpha} h(p)\\
\frac{1}{B(\alpha)}\bigg[ \frac{p^{\alpha} \big(B(\alpha)+(1-\alpha)c_{1} \big) + \alpha c_{1}}{p^{\alpha}} \bigg]u(p) = \frac{1}{p}u(0) + \frac{(1-\alpha)c_{0}}{B(\alpha)}h(p) + \frac{\alpha c_{0}}{B(\alpha)}p^{-\alpha} h(p) \\
u(p) = \frac{B(\alpha)}{p} \bigg[\frac{p^{\alpha}}{ p^{\alpha} \big(B(\alpha)+(1-\alpha)c_{1} \big) + \alpha c_{1}} \bigg]u(0) +  \bigg[ \frac{(1-\alpha)c_{0} p^{\alpha}}{ p^{\alpha} \big(B(\alpha)+(1-\alpha)c_{1} \big) + \alpha c_{1}} \bigg]h(p) \\
 + B(\alpha) \bigg[ \frac{p^{\alpha}}{ p^{\alpha} \big(B(\alpha)+(1-\alpha)c_{1} \big) + \alpha c_{1}} \bigg] \frac{\alpha c_{0}}{B(\alpha)}p^{-\alpha} h(p) \\
= \frac{B(\alpha)}{\big(B(\alpha)+(1-\alpha)c_{1} \big)} \bigg[\frac{p^{\alpha-1}}{ p^{\alpha} + \frac{\alpha c_{1}}{\big(B(\alpha)+(1-\alpha)c_{1} \big)}} \bigg]u(0)\\
 +  \frac{(1-\alpha)c_{0}}{\big(B(\alpha)+(1-\alpha)c_{1} \big)} \bigg[\frac{p^{\alpha}}{ p^{\alpha}+ \frac{\alpha c_{1}}{\big(B(\alpha)+(1-\alpha)c_{1} \big)}} \bigg]h(p) \\
 +\frac{\alpha}{\big(B(\alpha)+(1-\alpha)c_{1} \big)} \bigg[\frac{c_{0}}{ p^{\alpha}+ \frac{\alpha c_{1}}{\big(B(\alpha)+(1-\alpha)c_{1} \big)}} \bigg]h(p) 
\end{multline}
if we set
\[
\gamma = \frac{\alpha c_{1}}{\big( B(\alpha) + (1-\alpha)c_{1} \big)} ~~\text{and} ~~\zeta = \frac{B(\alpha)}{\big( B(\alpha) + (1-\alpha)c_{1} \big)}.
\]
it comes that

\begin{multline}
u(p) = \zeta~\frac{p^{\alpha-1}}{ p^{\alpha} + \gamma} ~u(0) 
+  \frac{c_{0}\zeta (1-\alpha)}{B(\alpha)}~\frac{p^{\alpha}}{ p^{\alpha}+ \gamma}~h(p) 
 +\frac{\alpha c_{0}\zeta}{B(\alpha)} ~\frac{1}{ p^{\alpha} + \gamma}~ h(p) \\
=\zeta~\frac{p^{\alpha-1}}{ p^{\alpha} + \gamma} ~u(0) +  \frac{\zeta c_{0}(1-\alpha)}{B(\alpha)}~\frac{1}{ 1 + \big(\gamma^{-\alpha}p\big)^{-\alpha}}~h(p) 
 +\frac{\alpha c_{0} \zeta}{B(\alpha)} ~\frac{1}{ p^{\alpha} + \gamma}~ h(p)  
\end{multline}
with $d = \gamma^{-\alpha}$. We notice that 
\[
\begin{aligned}
\frac{1}{ 1 + \big(\gamma^{-\alpha}p\big)^{-\alpha}} & = \sum _{k=1}^{\infty} \frac{d^{\alpha k-2}t^{\alpha k-1}}{\Gamma(\alpha k)} = \frac{d}{dt}E_{\alpha} \big[\gamma^{-2\alpha}  t^{\alpha}\big] = d^{-2\alpha}t^{\alpha-1}E_{\alpha,\alpha} \big[\gamma^{-2\alpha}t^{\alpha-1}\big]
\end{aligned}
\]
Hence by applying the inverse Laplace transform, we the result as
\[
\begin{aligned}
u(t) & = \zeta E_{\alpha}\big[-\gamma t^{\alpha}\big]u(0)\\
&+ \frac{\alpha c_{0} \zeta}{B(\alpha)} \int_{0}^{t}\bigg( E_{\alpha,\alpha}\big[-\gamma(t-s)^{\alpha}\big] + \frac{(1-\alpha)}{\alpha}\gamma^{-2\alpha}E_{\alpha,\alpha}\big[-\gamma^{-2\alpha}(t-s)^{\alpha}\big]\bigg)(t-s)^{\alpha-1}h(s)ds .
\end{aligned}
\]
\end{proof}

\begin{cor}\label{Collolary_a}
Let $g:[-2,0] \to \R$ be a solution to $\AB g = -c_{1}g + c_{0}h(t)$ with $g(-2) = 0$, $h\geq 0$ and $\int_{-2}^{-1}h(t) \geq \mu.$
Then 
\[
g(t) \geq \frac{\alpha}{2}E_{\alpha,\alpha}\big[-2 c_{1}\big] c_{0} \mu, \quad \text{for} \quad -1 \leq t\leq 0.
\]
\end{cor}
\begin{proof}
From \autoref{proposition_solution} the solution of the differential equation for $g$ can be computed explicitly
\[
\begin{aligned}
g(t) & = \zeta E_{\alpha}\big[-\gamma t^{\alpha}\big]g(-2)\\
&+ \frac{\alpha c_{0} \zeta}{B(\alpha)} \int_{-2}^{t}\bigg( E_{\alpha,\alpha}\big[-\gamma(t-s)^{\alpha}\big] + \frac{(1-\alpha)}{\alpha}\gamma^{-2\alpha}E_{\alpha,\alpha}\big[-\gamma^{-2\alpha}(t-s)^{\alpha}\big]\bigg)(t-s)^{\alpha-1}h(s)ds .
\end{aligned}
\]
If $c_{o}$ is small enough and $c_{1}$ is large, with the initial condition $g(-2) = 0$, and the fact that $E_{\alpha,\alpha}(t)>0$, we have 
\[
\begin{aligned}
g(t) & \geq \frac{\alpha c_{0} \zeta}{2 B(\alpha)}E_{\alpha,\alpha}\big[-2\gamma \big] \int_{-2}^{t}h(s)ds
 \geq \frac{\alpha c_{0} \zeta}{2 B(\alpha)}E_{\alpha,\alpha}\big[-2c_{1} \big]c_{0} \mu \\
& \geq \frac{\alpha}{2}E_{\alpha,\alpha}\big[-2c_{1} \big]c_{0} \mu.
\end{aligned}
\]
\end{proof}

\section{Existence of weak solutions via approximating solutions}\label{sec:Existence}
In this section we provide details of the proof of existence of a solution to the weak equation \eqref{equation:problem} via approximating solutions. We follow the idea of \cite{Allen1} to prove our result. To do this, we shall start with the weak formulation of the problem, then provide the discretization of the weak formulation that will enable us later to prove the existence of the unique solution and H\"{o}lder continuity.\medskip

We state the following  integration by parts which will be used to prove weak formulation.

\begin{prop}\label{Proposition_a}
Assume that $\vartheta $ is bounded and Lipschitz function on $(a,b)$, then
\begin{equation}\label{equation:proposition_a}
\begin{aligned}
\int_{a}^{b}\vartheta(t) \L u(t) dt & = c_{\alpha}\int_{a}^{b}u(t) \bigg(\int_{a}^{t}\T(t,s)\big[\vartheta(t)- \vartheta(s)\big]ds \bigg) dt \nonumber \\
& + c_{\alpha} \int_{a}^{b}\int_{a}^{t}\T(t,s)\big[u(t)-u(s)\big]\big[\vartheta(t)-\vartheta(s)\big]ds~dt \\
& - \int_{a}^{b}E_{\alpha}\big[c(t-a)^{\alpha} \big] \bigg[ u(t)\vartheta(a) + u(a)\vartheta(t) \bigg]~dt.\nonumber \\
&-\int_{a}^{b} u(t) \L \vartheta(t) dt
\end{aligned}
\end{equation}
\end{prop}
\begin{proof}
The proof of the \autoref{Proposition_a} follows from the direct computation of the first term before the equality.
\end{proof}
Using the definition of the Atangana-Baleanu fractional derivative in the form of \eqref{eq:Atangana_Baleanu_other_form}, and for $t<a$, we have the following relation

\begin{equation}\label{equation:weak_form1}
\begin{aligned}
\int_{-\infty}^{b}\vartheta(t) \L u(t) dt &= c_{\alpha}\int_{-\infty}^{b}\int_{-\infty}^{b}\big[u(t)-u(s)\big]\big[\vartheta(t)-\vartheta(s)\big] \T(t,s) ds~dt \nonumber \\
& + c_{\alpha}\int_{-\infty}^{b}\int_{-\infty}^{2t-b}
u(t)\big[\vartheta(t)- \vartheta(s)\big]\T(t,s)ds~dt \nonumber \\
& - \int_{-\infty}^{b}u(t)\L \vartheta(t)~dt.\nonumber
\end{aligned}
\end{equation}

Now we state the weak formulation of the problem \eqref{equation:problem}.\medskip

\begin{prop}\label{proposition_b}
Assume that $\vartheta$ is a bounded and Lipschitz function on $(-\infty,b)$ for any $t < a$ then the following weak formulation of the solutions~\eqref{equation:problem} holds
\begin{equation}
\begin{aligned}
& c_{\alpha} \int_{\R^{n}}\int_{-\infty}^{b}\int_{-\infty}^{t} \big[u(t,x)-u(s,x)\big]\big[\vartheta(t,x)-\vartheta(s,x)\big]\T(t,s,x)ds~dt~dx \nonumber \\
&+ \int_{a}^{b} \int_{\R^{n}}\int_{\R^{n}} \J(t,x,\xi)\big[u(t,x)-u(t,\xi) \big] \big[\vartheta(t,x)-\vartheta(t,\xi) \big] dx ~d\xi~dt \nonumber \\
&+ c_{\alpha} \int_{\R^{n}} \int_{-\infty}^{b}\int_{-\infty}^{2t-b} u(t,x)\vartheta(t,x)\T(t,s,x)ds~dt~dx \nonumber 
- \int_{\R^{n}}\int_{-\infty}^{b} u(t,x) \L \vartheta(t,x) dt~dx \\
&= \int_{-\infty}^{b} \int_{\R^{n}} f(t,x)\vartheta(t,x)dx~dt.\nonumber
\end{aligned}
\end{equation}
\end{prop}

\begin{proof}
Now to prove the statement of the \autoref{proposition_b} we use the equation~\eqref{Atangana-Baleanu_Caputo_other} and we follow the idea introduced in \cite{Allen}.\medskip

Let $\vartheta$ is a bounded and Lipschitz function on $(-\infty,b)$. Then
\begin{gather}\label{relation}
\int_{a}^{b}\int_{a}^{t}\T(t,s)\big[\vartheta(t,x)- \vartheta(s,x)\big]ds~ dt =  \int_{a}^{b}\int_{a}^{t}\T(t,s)\vartheta(t,x)- \int_{a}^{b}\int_{a}^{t}\T(t,s) \vartheta(s)ds~ dt \nonumber \\
 = \int_{a}^{b}\int_{s}^{b}\T(t,s)\vartheta(t)- \int_{a}^{b}\int_{a}^{t}\T(t,s) \vartheta(s)dt~ ds \\
 = \int_{a}^{b}\int_{a}^{t}\T(t,s)\vartheta(t)- \int_{a}^{b}\int_{t}^{b}\T(s,t) \vartheta(s)ds~ dt, \nonumber 
\end{gather}
thanks to~\eqref{eq:kernel_relation} the kernel $\K(t,s)$ satisfies that condition, therefore it can be written as
\begin{equation}\label{symetry_kernel}
\T(t,t-s) = \T(t+s,t) = s^{\alpha-1} E_{\alpha,\alpha}\big[c s^{\alpha} \big].
\end{equation}
Thus from ~\eqref{symetry_kernel}, the equation~\eqref{relation} becomes
\begin{equation}\label{preuve_next}
\begin{aligned}
& \int_{a}^{b}\int_{a}^{t}\T(t,s)\big[\vartheta(t)- \vartheta(s)\big]ds~ dt = \nonumber \\ 
&\int_{a}^{b} \vartheta(t) \bigg(\int_{0}^{t-a} s^{\alpha-1} E_{\alpha,\alpha}\big[c s^{\alpha} \big]ds 
-\int_{0}^{b-t} s^{\alpha-1} E_{\alpha,\alpha}\big[c s^{\alpha} \big]ds \bigg) dt \nonumber \\
& = \frac{1}{\alpha c_{\alpha}}\int_{a}^{b} \vartheta(t) \bigg( E_{\alpha}\big[c (b-t)^{\alpha} \big] - E_{\alpha}\big[c (t-a)^{\alpha} \big] \bigg) dt
\end{aligned}
\end{equation}
By making use of equation~\eqref{preuve_next} and \eqref{equation:proposition_a}, it follows that
\begin{equation}\label{equation: weak_formulation}
\begin{aligned}
& c_{\alpha} \int_{\R^{n}}\int_{-\infty}^{b}\int_{-\infty}^{t} \big[u(t,x)-u(s,x)\big]\big[\vartheta(t,x)-\vartheta(s,x)\big]\T(t,s,x)ds~dt~dx  \\
&+ \int_{a}^{b} \int_{\R^{n}}\int_{\R^{n}} \J(t,x,\xi)\big[u(t,x)-u(t,\xi) \big] \big[\vartheta(t,x)-\vartheta(t,\xi) \big] dx ~d\xi~dt  \\
&+ c_{\alpha} \int_{\R^{n}} \int_{-\infty}^{b}\int_{-\infty}^{2t-b} u(t,x)\vartheta(t,x)\T(t,s,x)ds~dt~dx  
- \int_{\R^{n}}\int_{-\infty}^{b} u(t,x) \L \vartheta(t,x) dt~dx \\
&= \int_{-\infty}^{b} \int_{\R^{n}} f(t,x)\vartheta(t,x)dx~dt.
\end{aligned}
\end{equation}
This completes the proof of the \eqref{proposition_b}.
\end{proof} 

\subsection{Discretization of the problem}
In the following, we denote by $\tau = b/\kappa$ the time step which represents the subdivision of the interval $(a,b)$, where $\kappa \in \N$ denotes the number of time steps. Also, for $0 \leq k \leq \kappa$, $t = k\tau$. So the discrete form of the Atangana-Baleanu fractional derivative in the sense of Caputo holds
\begin{equation}\label{equation:discrete_Atangana-Baleanu}
\L u(a + \tau k) = \tau^{\alpha} c_{\alpha} 
    \sum_{-\infty<i <k} \frac{E_{\alpha,\alpha}\big[c \tau^{\alpha} (k-i)^{\alpha} \big] \big[ u(a+ \tau k) - u(a+ \tau i) \big]}{\big(k-i\big)^{1-\alpha} }
 \end{equation}
Using \eqref{equation:discrete_Atangana-Baleanu} the discrete form of \eqref{Atangana_baleanu_def} takes the form 
\begin{equation}\label{equation:problem_discrete}
\begin{aligned}
\tau^{\alpha} c_{\alpha} 
\sum_{-\infty<i <k} \frac{E_{\alpha,\alpha}\big[c \tau^{\alpha} (k-i)^{\alpha} \big] \big[ u(a+ \tau k) - u(a+ \tau i) \big]}{\big(k-i\big)^{1-\alpha} } &= 
\int_{\R^{n}}[u(a+\tau k,\xi)-u(a+\tau k,x)]\\
& \J(a+\tau k,x,\xi)d\xi +  g(a+\tau k,x). \nonumber
\end{aligned}
\end{equation} 
Next we state the following Lemma 
\begin{lem} \label{Lemma_integration_by_parts_discrete}
Assume $u(a) = u(a+\varepsilon j) = 0$ for $j<0$. Then the discrete integration by parts type estimate holds 
\begin{equation}\label{equation:integration_by_parts_discrete}
\begin{aligned}
\sum_{k\leq j} u(a+\tau k)\L u(a+ \tau k)& \geq 
\frac{\tau^{\alpha}}{2} C_{\alpha} \mathop{\sum \sum}_{0\leq i < k \leq j} \frac{\big[ u(a+ \tau k) - u(a+ \tau i) \big]^{2}}{\big(k-i\big)^{\alpha-1}}E_{\alpha,\alpha}\big[c \tau^{\alpha} (k-i)^{\alpha} \big]  \\
& +\frac{\tau^{\alpha}}{2} C_{\alpha} \mathop{\sum}_{ k \leq j} \frac{\big[ u^{2}(a+ \tau k)\big]^{2}}{\big(j-i\big)^{\alpha-1}}E_{\alpha,\alpha}\big[c \tau^{\alpha} (j-i)^{\alpha} \big]  \\.
\end{aligned}
\end{equation}    
\end{lem}
\begin{proof}
The proof of this Lemma follows from the direct computation of the discrete integration by parts type estimate.
\end{proof}
Now to get the discrete form of the weak formulation \autoref{proposition_b} we use 
the integration by parts type estimate given by  \autoref{Lemma_integration_by_parts_discrete}

\begin{equation}\label{discrete_weak_formulation}
\begin{aligned}
& c_{\alpha} \mathop{\sum \sum}_{0\leq i<k\leq j} \int_{\tau (k-1)}^{\tau k} 
\int_{\tau (i-1)}^{\tau i} \big[u_{\tau}(\tau k)-u_{\tau}(\tau i) \big] \big[
\vartheta(\tau k)-\vartheta(\tau i)\big]\frac{ E_{\alpha,\alpha}\big[c \tau^{\alpha} (k-i)^{\alpha} \big]}{\big(k-i\big)^{1-\alpha}}. \\
&+ \int_{\tau (k-1)}^{\tau k} \int_{\R^{n}}\int_{\R^{n}} \J(t,x,\xi)\big[u(\tau k, \xi)-u(\tau k ,x) \big] \big[\vartheta(\tau k, x)-\vartheta(\tau k,\xi) \big]  \\
&+ c_{\alpha} \int_{\R^{n}} \tau \mathop{\sum \sum}_{0\leq i<k\leq j} \frac{u_{\tau}(\tau k)\vartheta(\tau k)}{(\tau(k-i))^{1-\alpha}}E_{\alpha,\alpha}\big[c \tau^{\alpha} (k-i)^{\alpha}\big] 
- \int_{\R^{n}}  \tau \sum_{0<k\leq j} u_{\tau}(\tau k) 
\L \vartheta(\tau k) \\
&= \int_{\tau (k-1)}^{\tau k} \int_{\R^{n}} g(\tau k,x)\vartheta(\tau k,x).
\end{aligned}
\end{equation}

Now we are ready to prove the \autoref{theo:existence} on existence of weak solutions following the idea of the authors in \cite{Allen} by using the approximating method. For this purpose, we write the operators in \eqref{equation: weak_formulation} and \eqref{equation:integration_by_parts_discrete} as $\mathcal{H}$ and $\mathcal{H}_{\tau}$ respectively.
 
\begin{proof}
For $\vartheta$ be a bounded an Lipschitz function on $(-\infty,b)\times \R^{n}$. There exists a sequence of solutions $u_{\tau}$ to \eqref{equation: weak_formulation} with $\tau \to 0$, such that 
\[
u_{\tau} \to u \in L^{p} \big( (-\infty,b) \times \R^{n}\big), 
\]
with $p$ as defined in \cite{Allen1} as
\[
p = 2 \bigg(\frac{\alpha n + \beta}{\alpha n + (1-\alpha) \beta} \bigg).
\] 
For $\tau (k-1) < t \leq \tau k$, we let $\mathcal{B}_{\tau}$ be the bilinear form associated with $K_{\tau}$. Our aim is to show that for $\vartheta$ be a bounded and Lipschitz function on $(-\infty,b)\times \R^{^n}$,
\[
\mathcal{H}(u,\vartheta)+ \mathcal{H}_{\tau}(u_{\tau},\vartheta) \to 0.
\]
The first part of the proof where we shall consider the fractional Laplacian, we mean
\[
\lim_{\tau \to 0} \int_{a}^{b} \mathcal{B}_{\tau}(u_{\tau},\vartheta) \ dt 
 = \lim_{\tau \to 0} \tau \sum_{0<k\leq j} \mathcal{B}
      (u_{\tau}(\tau k, x), \vartheta(\tau k,x))
\]
has been prove by the authors in \cite{Allen1}. So next we shall focus on the pieces in time. To do this we start by showing that
\[
\begin{aligned}
    &\lim_{\tau \to 0}\int_{\R^{n}}\int_{-\infty}^{b}\int_{-\infty}^{t}\big[u_{\tau}(t,x)-u_{\tau}(s,x)\big]\big[\vartheta(t,x)-\vartheta(s,x)\big]\frac{E_{\alpha,\alpha}\big[c (t-s)^{\alpha} \big]}{(t-s)^{1-\alpha}} ds~dt \\
&\quad = \lim_{\tau \to 0}
\mathop{\sum \sum}_{0\leq i<k\leq j} \int_{\tau (k-1)}^{\tau k} 
\int_{\tau (i-1)}^{\tau i} \big[u_{\tau}(\tau k)-u_{\tau}(\tau i) \big] \big[
\vartheta(\tau k)-\vartheta(\tau i)\big]\frac{ E_{\alpha,\alpha}\big[c \tau^{\alpha} (k-i)^{\alpha} \big]}{\big(\tau(k-i)\big)^{1-\alpha}}.  \\
\end{aligned}
\]
In order to achieve our goal, since $\vartheta$ is a bounded and Lipschitz function, then $\vartheta_{\tau}(t) \to \vartheta(t)$ and $u_{\tau} \to u \in L^{p} \big( (-\infty,b) \times \R^{n}\big)$ we have that
\[
\lim_{\tau \to 0} \left| \int_{\R^n} \int_{-\infty}^b \int_{\infty}^t \frac{E_{\alpha,\alpha}\big[c (t-s)^{\alpha} \big]}{(t-s)^{1-\alpha}}\big[u_{\tau}(t)-u_{\tau}(s)\big]\left[(\vartheta(t)-\vartheta(s))
-(\vartheta_{\tau}(t)- \vartheta_{\tau}(s))\right] \ ds \ dt \right| \to 0.
\]
We show also that
\begin{equation}  \label{e:1/2bound}
\begin{aligned}
& \lim_{\tau \to 0} \mathop{\sum \sum}_{0\leq i<k\leq j} (u_{\tau}(\tau k)-u_{\tau}(\tau i))
(\vartheta(\tau k)-\vartheta(\tau i))  \\
&\times\int_{\tau (k-1)}^{\tau k} \int_{\tau (i-1)}^{\tau i}
\bigg( \frac{E_{\alpha,\alpha}\big[c \big(t - s\big)^{\alpha} \big]}{ (t-s)^{1-\alpha}} - \frac{E_{\alpha,\alpha}\big[c \tau^{\alpha}\big(k - i\big)^{\alpha} \big]}{\big(\tau(k-i) \big)^{1-\alpha}} \bigg) = 0.\nonumber         
\end{aligned}
\end{equation}   
To do this we break up the integral over the sets $(t-s)\leq \tau^{1/2}$ and $(t-s)> \tau^{1/2}$. Then with the relation ~\eqref{eq:kernel_relation}
\[
\begin{aligned}
0 &= \lim_{\tau \to 0}  \int_{\R^n} \mathop{\int \int}_{t-s\leq \tau^{1/2}}
\frac{|(u_{\tau}(t)-u_{\tau}(s)(\vartheta_{\tau}(t)-\vartheta_{\tau}(s))|}{(t-s)^{1-\alpha}} E_{\alpha,\alpha}\big[c\big(t- s\big)^{\alpha} \big]\\
&\geq \lim_{\tau \to 0} \int_{\R^n} \mathop{\sum \sum}_{\tau (k-i) \leq \tau^{1/2}}
\bigg|(u_{\tau}(\tau k)-u_{\tau}(\tau i))(\vartheta(\tau k)-\vartheta(\tau i))\bigg| 
\int_{\tau (k-1)}^{\tau k} \int_{\tau (i-1)}^{\tau i} 
\frac{E_{\alpha,\alpha}\big[c \tau^{\alpha}\big(k - i\big)^{\alpha} \big]}{(\tau(k-i))^{1-\alpha}}
\end{aligned}
 \]
For $t-s>\tau^{1/2}$, and $\tau (i-1)\leq s \leq \tau i$ and $\tau (k-1)\leq t \leq \tau k$, we can compute the estimate
\[
\begin{aligned}
\big| (t-s)^{\alpha-1}E_{\alpha,\alpha}\big[c\big(t- s\big)^{\alpha} \big] - \big(\tau(k-i)\big)^{\alpha-1} \big| 
&\leq \big(\tau^{1/2}-\tau\big)^{\alpha-1}  \\
&\times E_{\alpha,\alpha}\big[c\big(\tau^{1/2}-\tau\big)^{\alpha} \big] - \big(\tau^{1/2}\big)^{\alpha-1}E_{\alpha,\alpha}\big[c\big(\tau^{1/2}\big)^{\alpha} \big]\\
&\leq T_{\alpha,\tau}
\end{aligned}
\]
Hence the result follows for $\tau \to 0$. 
Next we consider the following    
\[
\int_{\R^n} \int_{-\infty}^b u_{\tau}(t) \L \vartheta(t) \ dt \ dx - 
\int_{\R^n}  \tau \sum_{0<k\leq j} u_{\tau}(\tau k) 
\L \vartheta(\tau k).   
\]
Similarly as in the previous case, since $\vartheta$ is a bounded Lipschitz function and $u_{\tau} \to u \in L^{p} \big( (-\infty,b) \times \R^{n}\big)$, and from  \eqref{equation:integration_by_parts_discrete} one show that this term also goes to zero.\medskip

The remaining pieces in time are handled in the similar manner. Thus the Theorem is proved.  
\end{proof}

\section{Pointwise Estimates and H\"{o}lder Regularity}\label{sec:PointWiseEvaluation}


This section contains the some auxiliaries results, which are the key to prove \autoref{thm:holder}. The proof of \autoref{point_estimate} uses the main contributions of this note. Once \autoref{point_estimate} is established, a-priori H\"{o}lder regularity estimates follow by the classical method of diminishing oscillation given by \autoref{diminish_osc}.

Before going into this, we first collect the ingredients that will be useful. As one feature, we underline the viscosity solution. One of the useful property of viscosity solution is that viscosity subsolutions themselves can be used to evaluate their corresponding equation classically at all of the points where  where the subsolution can be touched from above by a smooth test function. \medskip

We need the following property
\begin{equation}\label{PointwiseEq:USubsolution}
\L u(t,x) - M^{+}_{\A}u(t,x)\leq g(t,x)\ \quad \text{in}\  B_1 \times [-1,0].
\end{equation}
Next we state the following proposition to clarify that $u$ is a solution on  \eqref{PointwiseEq:USubsolution} and \eqref{equation:problem} in the viscosity sense by also making reference to the \autoref{continuous_bounded_derivative} 
\begin{prop}\label{continuous_bounded_function}
Let $u$ be a continuous bounded function on $(-\infty, b)$ and assume that for some $t \in (-\infty,b)$ there is a Lipschitz function touching $u$ by above at $t$. Then 
\[
\int_{-\infty}^t \big[u(t,x)- u(s,x)\big] \T(t,s) \ ds \geq g(t,x)
\]
if and only if $\L u(t,x) \geq g(t,x)$ in the viscosity sense.
\end{prop}

\begin{proof}
The proof is standard and is based on the proof of Proposition 2.3 in \cite{Allen2}.
\end{proof}

\begin{prop}\label{prop:PointwiseEvaluation}
Assume $u$ solves \eqref{PointwiseEq:USubsolution} in the viscosity sense. $\phi \geq u$ defined on the cylinder $\mathcal{C}:=[t_{0}-\varepsilon,t_{0}] \times B_{\varepsilon}(x_{0})$ has a global maximum and touches $u$ from above at $(x_{0},t_{0}) \in \mathcal{C}$ and we define $v$ as 
\[ 
v(x,t) := 
\begin{cases}
\phi(x,t) & \text{if } (x,t) \in \mathcal{C} \\ 
u(x,t) & \text{otherwise},\\
\end{cases}
\]
then $v$ is solution to \eqref{PointwiseEq:USubsolution} at $(x_{0},t_{0})$ or
\[
\L v(t_{0},x_{0}) - \mathcal{J}v(t_{0},x_{0}) \leq  g(t_{0},x_{0}).
\]
So the solution is both a subsolution and supersolution.\medskip 
\end{prop}

The proof is straightforward.\medskip

Before we state the point evaluation proposition, we recall the comparison principle for which the proof is similar as in~ \cite{Allen2}.
\begin{lem}\label{Comparison_principle}
(Comparison Principle). Let $u$ be bounded and upper simi-continuous and $w$ be bounded and lower semi-continuous $(-\infty, b_{2})$. Let $g$ be a continuous function such that $\L u \leq g \leq \L w$ on $(b_{1},b_{2}]$, with $u \leq w$ on $(-\infty, b_{1})$. Then $u \leq w$ on $(-\infty, b_{2})$, and if $u(t_{0}) = w(t_{0})$ for some $t_{0} \in (b_{1},b_{2}]$, then $u(t) = w(t)$ for all $t \leq t_{0}$.
\end{lem}

\begin{lem}\label{viscosity_solution_of_continuous_function}
Let $\phi(t)$ be continuous $(-\infty, b_{1})$. Let $g$ be a continuous function on $[b_{1}, b_{2}]$. There exists a unique viscosity solution $u$ to 
, then 
\[
\begin{cases}
\L u(t) &= g(t) \qquad \text{for} \quad \quad ~t \in [b_{1}, b_{2}]\\

~~~u(t) &= \phi(t) \qquad \text{if} \quad \quad  t \leq b_{1}
\end{cases}
\]
on $(b_{1},b_{2}]$.
\end{lem}

We start by recalling the definitions of Pucci's extremal operators as defined in \cite{Kassmann} for the spatial operator and next for the fractional-time derivative operator.
\begin{lem}[Extremal Formula]\label{lem:FormulaForMPlus}
Assume $u\in C^{1,1}(-\infty,b)\intersect L^\infty(\R^n)$. Then we have the following elliptic spacial operator 
\begin{equation}\label{PointEq:FormulaMPlus}
\M^{+}_{\A}u(t,x) = \int_{\R^n} \bigg(\Lambda\big(\delta_h u(t,x)\big)_{+} - \lambda\big(\delta_h u(t,x)\big)_{-} \bigg)\mu(dh),
\end{equation}
\begin{equation}\label{PointEq:FormulaMMinus}
\M^{-}_{\A}u(t,x) = \int_{\R^n} \bigg(\lambda\big(\delta_h u(t,x)\big)_{+} - \Lambda\big(\delta_h u(t,x)\big)_{-} \bigg)\mu(dh),
\end{equation}
and for the fractional-time derivative operator as
\begin{equation}\label{Puccifractionaltime1}
\M_{\T_{sec}}^{+} u(t,x) := \mathcal{C}(n,\alpha)\int_{-\infty}^{t} \bigg[ \Lambda \big(u(t,x) - u(s,x)\big)_{+} - \lambda \big(u(t,x) - u(s,x)\big)_{-} \bigg]\T_{sec} 
\end{equation}
\begin{equation}\label{Puccifractionaltime2}
\M_{\T_{sec}}^{-} u(t,x) := \mathcal{C}(n,\alpha)\int_{-\infty}^{t} \bigg[ \lambda \big(u(t,x) - u(s,x)\big)_{+} - \Lambda \big(u(t,x) - u(s,x)\big)_{-} \bigg]\T_{sec} \end{equation}
\end{lem} 

Next, in order to prove the H\"{o}lder continuity, we use essentially the same ideas as the proof in \cite{Allen, Luis_sylvestre}.\medskip

\begin{prop}\label{point_estimate}
(Point Estimate). Let $u \leq 1$ in $(\R^{n} \times [-2, 0]) \cup (B_{1} \times [-\infty, 0])$
 and assume it satisfies the following inequality in the viscosity sense in $B_{2} \times [-2, 0]$
\[
\L u - \M^{+}u \leq \varepsilon_{0}.
\]
Assume also that $\big| \big\{u(x,t) \leq 0 \big\} \cap (B_{1} \times [-2,-1]) \big| \geq \mu > 0 $.

Then  if $\varepsilon_{0}$ is small enough there exists $\theta >0$ such that $u \leq (1-\theta)$ in $B_{1} \times [-2,0]$. The maximal value of $\theta$ as well as $\varepsilon_{0}$ depends on $\alpha, \lambda, \Lambda, n$ and $\sigma$, but remain uniform as $\alpha \to 1$.
\end{prop}

\begin{proof}
We consider the differential equation
\begin{equation}\label{visco_aaa}
\begin{cases}
\L g(t) = c_{0}\big| \big\{x \in B_{1}:u(x,t) \leq 0 \big\} \big| - c_{1}g(t).\\

g(-2) = 0, \quad \text{for}~~ t \leq -2
\end{cases}
\end{equation}
From \autoref{proposition_solution} this ordinary differential equation can be computed explicitly and from \autoref{Collolary_a} we have that 
\[
g(t) \geq \frac{\alpha c_{0} \mu}{2}E_{\alpha,\alpha}\big[-2c_{1} \big], \qquad \text{for} -1 \leq t \leq 0. 
\]
In the following we will show that if $c_{0}$ is small and $c_{1}$ is large, then $u \leq 1 - g(t) + \varepsilon_{0} c_{\alpha}2^{\alpha}$ in $B_{1} \times [-1,0]$. The constant $c_{\alpha}$ is chosen such that $\L c_{\alpha} (2+t)^{\alpha}_{+} = 1$ for $t \geq -2$.\medskip

Since for $t \in [-1,0]$
\[
\begin{aligned}
g(t) & \geq \frac{\alpha}{2}E_{\alpha,\alpha}\big[-2c_{1} \big]c_{0} \big|\big\{x:u(x,t)\leq 0 \big\} \cap B_{1} \times [-2,-1] \big|\\
& \geq  \frac{\alpha}{2}E_{\alpha,\alpha}\big[-2c_{1} \big]c_{0} \mu,
\end{aligned}
\]
we set $\theta = \frac{\alpha c_{0} \mu}{4} E_{\alpha,\alpha}\big[-2c_{1} \big]$ for $\varepsilon_{0}$ small and finish the proof of the Lemma.\medskip

Next as in \cite{Allen}, let $\beta: \R \to \R$ be a fixed smooth non increasing function such that $\beta(x) = 0$ if $x \geq 2$. Let $\eta (x,t) = \beta(|x|)$. As a function of $x$, $\eta(x,t)$ looks like a bump function for every fixed $t$. The main strategy of the proof is to show that the function $u(x,t)$ stays below $1-g(t)\eta(x,t)+\varepsilon_{0}c_{\alpha}(2+t)_{+}^{\alpha}$.\medskip

In order to arrive to a contradiction, we assume that $\eta(x,t) > 1 - g(t)+\varepsilon_{0}c_{\alpha}(2+t)_{+}^{\alpha}$ for some point $(x,t) \in B_{1} \times [-1,0]$. We then look at the maximum of the function
\[
\tilde{\eta}(x,t) = u(x,t) + g(t)\eta(x,t) - \varepsilon_{0}c_{\alpha}(2+t)_{+}^{\alpha}.
\] 

Assuming that there is one point $(x_{0},t_{0})$ in $B_{1} \times [-1,0)$ where $\tilde{\eta}(x,t)  > 1$, $\tilde{\eta}$ must be larger that $1$ at the point that realize the maximum of $\tilde{\eta}$. We mean by that 
\[
\tilde{\eta}(x_{0},t_{0}) = \max_{\R^{n} \times (-\infty,0]} \tilde{\eta}(x,t).
\]
Since $\tilde{\eta}(x_{0},t_{0}) >1$, the point $(x_{0},t_{0})$ must belong to the compact support $\eta$. Hence $|x|<2$. Now on the remain of the proof is exactly as in \cite{Allen}.

Next we call $\varphi(x,t)$ the function such that 
\[
\varphi(x,t):= \tilde{\eta}(x_{0},t_{0}) - g(t)\eta(x,t) + c_{\alpha}(2+t)_{+}^{\alpha}
\]
$\varphi$ touches $u$ from above at the point $(x_{0},t_{0})$. We define
\[
v(x,t):=
\begin{cases}
\varphi(x,t) \quad \text{if}~~ x \in B_{r} \\
u(x,t) \quad \text{if}~~ x \notin B_{r}.
\end{cases}
\]
Then at the point that $\tilde{\eta}$ realise it maximum
\begin{equation}\label{visco_a}
\L v(x_{0},t_{0}) - \M ^{+}v(x_{0},t_{0}) \leq \varepsilon_{0}.
\end{equation}
So one can have 
\begin{equation}\label{visco_b}
\L v(x_{0},t_{0}) - \L g(t_{0})\eta(x_{0}) + \varepsilon_{0}.
\end{equation}
Hence for $G:=\big \{ x \in B_{1} \big|u(x,t_{0}) \leq o \big\}$, the following bound is obtain 
\begin{equation}\label{visco_c}
\M^{+}v(x_{0},t_{0})  \leq -g(t_{0}) \M^{-} \eta(x_{0},t_{0})-c_{0} \big|G \setminus B_{r} \big|
\end{equation}
Now if we insert the relation \eqref{visco_c}, \eqref{visco_a}, \eqref{visco_b} into \eqref{visco_aaa}, we obtain
\[
-~\L g(t_{0})\eta (x_{0},t_{0}) + \varepsilon_{0} + c_{0} \big|G\setminus B_{r}\big| \leq \varepsilon_{0},
\]
or in its explicit form
\[
\bigg(-c_{0} \big|\big\{x\in B_{1}:u(x,t) \leq 0 \big\} \big| + c_{1}g(t) \bigg) \eta(x_{0}) + \varepsilon_{0} + c_{0} \big|G\setminus B_{r}\big| \leq \varepsilon_{0}.
\]
But notice that for any $c_{1} > 0$ this contradicts \eqref{visco_aaa}.

We now analyse the case where $\eta(x_{0},t_{0}) > \beta_{1}$. As we said previously that $\eta$ is a smooth compactly supported function, then there exist some constant $\tilde{C}$ such that $\big| \M^{-}\eta \big| \leq \tilde{C}$. Then we have the bound 
\begin{equation}\label{visco_aaaa}
\M^{+}v(x_{0},t_{0}) \leq - \tilde{C}g(t_{0}) + c_{0} \big|G\setminus B_{r} \big|. 
\end{equation}

As in the previous case, we insert \eqref{visco_aaaa} in \eqref{visco_a}, \eqref{visco_b} into \eqref{visco_aaa}, we obtain
\[
\bigg(-c_{0} \big|\big\{x\in B_{1}:u(x,t) \leq 0 \big\} \big| + c_{1}g(t) \bigg) \eta(x_{0}) + \varepsilon_{0} - \tilde{C}g(t_{0}) + c_{0}\big|G\setminus B_{r}\big| \leq \varepsilon_{0}.
\]
We recall that $\eta(x_{0},t_{0}) >0$ and by letting $r \to 0$, we obtain 
\[
c_{0} \big(1-\eta(x_{0}) \big) |G| + \big( c_{1}\eta(x_{0})-\tilde{C}\big)g(t_{0}) \leq 0
\]
Which can be written in the form

\[
c_{0} \big(1-\eta(x_{0}) \big) |G| + \big( c_{1}\beta_{1}-\tilde{C}\big)g(t_{0}) \leq 0.
\]
Choosing $c_{1}$ large enough, we arrived to a contradiction. This ends the proof.
\end{proof}

In the following we shall give an approach of the proof of the H\"{o}lder continuity. For this purpose, we state and prove the so called growth Lemma, which say that if a solution of the equation \eqref{equation:problem} in the unit cylinder $Q_{1} = B_{1} \times [-1, 0]$ has oscillation one then its oscillation in a smaller cylinder $Q_{r}$ is less that a fixed constant $(1-\theta)$ \cite{ Allen, Luis_sylvestre}.\medskip

\begin{lem}\label{diminish_osc}
(Diminish of Oscillation). Let $u$ be a bounded continuous function which satisfies \eqref{equation:problem} or the following two inequalities in the viscosity sense in $Q_{1}$.
\[
\begin{aligned}
&\L u - \M^{+}u \leq \varepsilon_{0}/2 \\
&\L u - \M^{+}u \geq -\varepsilon_{0}/2.
\end{aligned}
\]
Then there are universal constants $\theta>0$ and $\nu>0$, depending on $n,\sigma,\Lambda,\lambda$ and $\alpha$, such that if 
\[
\begin{aligned}
\big\|g \big \|_{L^{\infty}(Q)} &\leq \varepsilon_{0}, \qquad |u| \leq 1 ~~ \text{in}~~ Q_{1} \\
|u(x,t)| &\leq 2|rx|^{\nu} - 1 ~~ \text{in}~~ \big(\R^{n} \setminus B_{1} \big) \times [-1, 0] \\
|u(x,t)| &\leq 2|rt|^{\nu} - 1 ~~ \text{in}~~ B_{1} \times [-\infty, -1] \\
\end{aligned}
\]
with $r=\min\{4^{-1},4^{-\alpha/2\sigma}\}$, then 
\[
\text{osc}_{Q_{r}}u \leq (1-\theta)
\]
\end{lem}
\begin{proof}
With the \autoref{proposition_estimate} in hand the reader can finish the proof by following the idea of the authors in \cite{ Allen, Luis_sylvestre}.
\end{proof}

Having the diminish of oscillation Lemma in hand, where the proof is similar to the one proposed by the author Mark Allen in \cite{Allen2}, we are now going to prove one of our main results about H\"{o}lder continuity. The result requires the function to solve the equation only in a cylinder in order to have H\"{o}lder continuity in a smaller cylinder. To do this, we renormalized the function $u$.\medskip

If $u$ satisfies \eqref{equation:problem}, then the rescaled function $\varrho(x,t) = u(rx,r^{2\sigma/\alpha}t)$ satisfies also \eqref{equation:problem}, with  $r \in (0,1)$ such that $r=\min\{4^{-1},4^{-\alpha/2\sigma}\}$. This means that the structure of the space is not modified if we make a parabolic dilation to make the profile a little bit in the future. We then define the parabolic cylinders in terms of the scaling of the equation by
\[
Q_r := B_r \times [-r^{2\sigma/\alpha},0].
\]
  
Now we state the result on H\"{o}lder continuity.\medskip

\begin{proof}[\autoref{thm:holder}]
 The proof is the adaptation of the proof in \cite{ Allen2, Luis_sylvestre}. We prove $\C^{\alpha}$ estimate of \eqref{c_alpha_estimate} by proving a $\C^{\alpha}$ estimate for $\varrho(x,t)$  at the point $(x_{0},t_{0}) \in Q_{1}$. For any point  $(x_{0},t_{0})$ we consider the normalized function   
\[
\varrho(x,t) = \frac{u(x_0 + x,t+t_0)}{\| u\|_{L^{\infty}} + \varepsilon_0^{-1} \| g\|_{L^{\infty}}},
 \]
where $\varepsilon_{0}$ is a constant from \autoref{diminish_osc}, one can show that $\text{osc}_{\R \times [-1,0]}\rho\leq 1$ and in $B_{2} \times [-1,0]$, $\varrho$ is a solution to 
   \[
    \begin{aligned} 
     \L \varrho - \M_{\sigma}^+\varrho &\leq \varepsilon_{0} \\
     \L \varrho - \M_{\sigma}^-\varrho &\geq -\varepsilon_{0}.
    \end{aligned}
   \]
Let $r \in (0,1)$, such that $r=\min\{4^{-1},4^{-\alpha/2\sigma}\}$ . To prove then the H\"{o}lder estimate, we prove by induction the following decay of the oscillation in cylinders for some $\kappa$ \cite{Luis_sylvestre1} 
   \begin{equation}  \label{equation:oscillation}
    \text{osc}_{Q_{r_k}} \varrho \leq 2 r^{\kappa k} \quad \text{for}~ k = 0,1,2,\dots.
   \end{equation}
We construct two sequences $S_k \leq \varrho \leq T_k$ in $Q_{r_k}$, $T_k - S_k= 2r^{\kappa k}$ with $S_k$ nondecreasing and $T_k$ nonincreasing. Indeed. for $k=0$, by hypothesis it is true that $\text{osc}_{Q_{r_k}} \varrho = 1$. We assume that the sequences holds until certain value of $k$.\medskip

We scale once more by considering 
   \[
    z(x,t) = (\varrho(r^k x,r^{2k\sigma /\alpha}t) -(S_k+T_k)/2)r^{-\kappa k}.
   \] 
Then we have 
   \[
    \begin{aligned}
     |z| \leq 1                 &\quad \text{in } \quad  Q_1 \\
     |z| \leq 2r^{-\kappa k} -1 &\quad \text{in } \quad  Q_{r^{-k}}.
    \end{aligned}
   \]
  and so
   \[
    \begin{aligned}
     |z(x,t)|\leq 2|x|^{\nu}-1  &\quad \text{for } \quad (x,t) \in B_1^c \times [-1,0] \\
     |z(x,t)|\leq 2|t|^{\nu}-1  &\quad \text{for } \quad (x,t) \in B_1 \times (-\infty,-1).
    \end{aligned}
   \]
For $\kappa < 2 \sigma$, $z$ has right hand side bounded 
\[
\big\|r^{k(\kappa-2\sigma)} g \big\|_{L^{\infty}} \leq  \varepsilon_0  r^{k(\kappa-2\sigma)} < \varepsilon_{0}
\]
If $\kappa$ small enough we can apply \autoref{diminish_osc} to obtain 
\[
\text{osc}_{ Q_r}  z \leq 1-\theta. 
\]
Furthermore if we choose $\kappa$ smaller than the one in \autoref{diminish_osc} and also so that $r^{\kappa} \geq 1-\theta $, we have that 
\[
\text{osc}_{ Q_{r^{k}}} z \leq r^{\kappa}
\]
which simply means $\text{osc}_{ Q_{r^{k+1}}} z \leq r^{\kappa(k+1)}$ so we can find two sequences $S_{k+1}$ and $T_{k+1}$. This finishes the proof by induction.  
 \end{proof}

\bibliographystyle{plain}

\end{document}